\documentclass[12pt]{amsart}
\usepackage[sans]{dsfont}
\usepackage{amsfonts,amssymb,amsbsy,amsmath,amsthm,enumerate}
\usepackage{color}

\numberwithin{equation}{section}

\newtheorem{theorem}{Theorem}[section]

\newtheorem{lemma}{Lemma}[section]
\newtheorem{follow}{Corollary}[section]

\newtheorem{pr}{Proposition}[section]
\theoremstyle{definition}

\def\beq{\begin{equation}}
\def\eeq{\end{equation}}
\newcommand{\bea}{\begin{eqnarray}}
\newcommand{\eea}{\end{eqnarray}}
\newcommand{\beas}{\begin{eqnarray*}}
\newcommand{\eeas}{\end{eqnarray*}}

\newcommand{\one}{\mathds{1}}
\newcommand{\bel}{\begin{equation} \label}
\newcommand{\ee}{\end{equation}}

\newcommand{\supp}{{\text{supp}}}
\newcommand{\rd}{{\mathbb R}^{2}}
\newcommand{\re}{{\mathbb R}}

\newcommand{\C}{{\mathbb C}}
\newcommand{\N}{{\mathbb N}}
\newcommand{\Z}{{\mathbb Z}}
\newcommand{\ttv}{T_V}
\newcommand{\harmo}{{\mathcal H}(\Omega)}

\newcommand{\eps}{{\varepsilon}}
\newcommand{\eeps}{{\epsilon}}

\begin{document}
\title[Harmonic Toeplitz Operators and Krein Laplacian]{Spectral Properties of Harmonic Toeplitz Operators\\
and Applications to the Perturbed Krein Laplacian}

\author[V. Bruneau]{Vincent Bruneau}
\address{Institut de Math\'ematiques de Bordeaux, UMR 5251 du CNRS,
Universit\'e de Bordeaux, 351 cours de la Lib\'eration, 33405 Talence cedex, France}
\email{vbruneau@math.u-bordeaux.fr}
\author[G. Raikov]{Georgi Raikov}
\address{Departamento de Matem\'aticas, Facultad de Matem\'aticas, Pontificia Universidad
Cat\'olica de Chile, Vicu\~na Mackenna 4860, Santiago de Chile}
\email{graikov@mat.uc.cl}

\begin{abstract}
We consider harmonic Toeplitz operators $T_V = PV : \harmo \to \harmo$ where $P : L^2(\Omega) \to \harmo$ is the orthogonal projection onto
$\harmo = \left\{u \in L^2(\Omega) \, | \, \Delta u = 0 \; \mbox{in} \; \Omega \right\}$, $\Omega \subset \re^d$, $d \geq 2$, is a bounded domain with boundary $\partial \Omega \in C^\infty$, and $V : \Omega \to \C$ is an appropriate multiplier. First, we complement the known criteria which guarantee that $T_V$ is in the $p$th Schatten-von Neumann class $S_p$,  by simple sufficient conditions which imply $T_V \in S_{p, {\rm w}}$, the weak counterpart of $S_p$. Next, we consider symbols $V \geq 0$ which have a regular power-like decay of rate $\gamma > 0$ at $\partial \Omega$, and we show that $T_V$ is unitarily equivalent to a pseudo-differential operator of order $-\gamma$, self-adjoint in $L^2(\partial \Omega)$. Utilizing this unitary equivalence, we obtain the main asymptotic term of the eigenvalue counting function for  $T_V$, and establish a sharp remainder estimate. Further, we assume that $\Omega$ is the unit ball in $\re^d$, and $V = \overline{V}$ is compactly supported in $\Omega$, and investigate the eigenvalue asymptotics of the Toeplitz operator $T_V$.  Finally, we introduce the Krein Laplacian $K$, self-adjoint in $L^2(\Omega)$, perturb it by a multiplier $V \in C(\overline{\Omega};\re)$, and show that $\sigma_{\rm ess}(K+V) = V(\partial \Omega)$. Assuming that $V \geq 0$ and $V_{|\partial \Omega} = 0$, we study the asymptotic distribution of the discrete spectrum of $K \pm V$ near the origin, and find that the effective Hamiltonian which governs this distribution is the Toeplitz operator $T_V$.
\end{abstract}

\maketitle

effective Hamiltonian\\


\section{Introduction}
\label{s1}
Let $\Omega \subset \re^d$, $d \geq 2$, be a bounded domain, i.e. a bounded open, connected, non-empty set. Suppose that $\partial \Omega \in C^\infty$.
Let ${\mathcal H}(\Omega)$ be the subspace
of $L^2(\Omega)$ consisting of functions harmonic in $\Omega$, i.e.
    \bel{1}
{\mathcal H}(\Omega) : = \left\{u \in L^2(\Omega) \, | \,  \Delta u = 0 \; \mbox{in} \; \Omega \right\}.
    \ee
It is well known that ${\mathcal H}(\Omega)$ is a closed subspace of  $L^2(\Omega)$ (see e.g. \cite{kako}). Let $P: L^2(\Omega) \to L^2(\Omega)$ be the orthogonal projection onto ${\mathcal H}(\Omega)$.
Assume that $V : \Omega \to {\mathbb C}$ is locally integrable in $\Omega$, and satisfies certain regularity conditions near $\partial \Omega$.
Then it can happen that the operator $T_V : = PV : \harmo \to \harmo$ called {\em harmonic Toeplitz operator with symbol $V$}, is bounded or even compact.
The article is devoted mostly to the study of the spectral properties of compact $T_V$.\\
First, in Section \ref{ss21} we recall some known criteria for the boundedness of $T_V$, its compactness,  and its membership to the Schatten-von Neumann classes $S_p$. Moreover, in Section \ref{ss21} we establish simple sufficient conditions which guarantee $T_V \in S_{p, {\rm w}}$, the weak Schatten-von Neumann class.\\ In Section \ref{ss23}, we assume that $V$ has a  power-like decay at $\partial \Omega$, and  establish in Proposition \ref{p9} a unitary equivalence between $T_V$ and  a certain pseudo-differential operator  acting in $L^2(\partial \Omega)$. We apply these results  in order to investigate in Theorem \ref{t1} the asymptotic distribution of the discrete spectrum of $T_V$.\\
Further, in Section \ref{symmetry} we consider the special case where $\Omega$ is the unit ball in $\re^d$.  If $V$ is radially symmetric, then the  eigenvalues and the eigenfunctions of $T_V$ could be written explicitly. Using these explicit calculations, we obtain the main asymptotic term of the eigenvalue counting function for $T_V$ for compactly supported $V$ with radially symmetric ${\rm supp}\,V$ (see Proposition \ref{p6}).\\
Finally, in Section \ref{s3} we introduce the Krein Laplacian $K$, self-adjoint in $L^2(\Omega)$. We have $K \geq 0$,  ${\rm Ker}\,K = {\mathcal H}(\Omega)$, and the zero eigenvalue of $K$ is isolated  (see \cite{k1, g1, as}). We perturb $K$ by the real-valued multiplier $V \in C(\overline{\Omega})$ and show that $\sigma_{\rm ess}(K+V) = V(\partial \Omega)$.  If $V \geq 0$ and $V_{| \partial \Omega} = 0$, we show that, generically, there exists a sequence of negative (resp., positive) discrete eigenvalues of the operator $K - V$ (resp., $K + V$), which accumulate to the origin from below (resp., from above). We show that the effective Hamiltonian governing the asymptotics of these sequences is the harmonic Toeplitz operator $T_V$ (see Theorem \ref{t3}). Using the results of the previous sections we obtain results on the eigenvalue asymptotics for the operators $K \pm V$ (see Corollaries \ref{f3} and \ref{f2}).

\section{Compactness and membership to Schatten-von Neumann Classes of harmonic Toeplitz operators $T_V$}
\label{ss21}
In this section we recall some known criteria for the boundedness, compactness and membership to the Schatten-von Neumann classes $S_p$, $p \in [1,\infty)$, of the
harmonic Toeplitz operator $T_V$, which we borrow mainly from \cite{cln}. Moreover, we establish simple sufficient conditions which guarantee $T_V \in S_{p,{\rm w}}$, $p \in (1,\infty)$, where $S_{p,{\rm w}}$ is the $p$th weak Schatten-von Neumann class.
\subsection{Notations}
First, we introduce the notations we need. Let $X$ and $Y$ be separable Hilbert spaces. We denote by ${\mathcal L}(X,Y)$ (resp., $S_\infty(X,Y)$) the class of linear bounded (resp., compact) operators $T: X \to Y$. Let $T \in S_\infty(X,Y)$. Then $\left\{s_j(T)\right\}_{j=1}^{{\rm rank}\,T}$ is the set of the non-zero singular values of $T$, enumerated in non-increasing order. Next, $S_p(X,Y)$, $p \in (0,\infty)$, is $p$th Schatten-von Neumann class, i.e. the class of compact operators $T: X \to Y$ for which the functional
    $$
    \|T\|_p : = \left(\sum_{j=1}^{{\rm rank}\, T} s_j(T)^p\right)^{1/p}
     $$
     is finite. Similarly, $S_{p,{\rm w}}(X,Y)$, $p \in (0,\infty)$, is the $p$th weak Schatten-von Neumann class, i.e. the class of operators $T \in S_\infty(X,Y)$ for which the functional
    $$
    \|T\|_{p, {\rm w}} : = \sup_{j \geq 1} j^{1/p} s_j(T)
     $$
     is finite. If $X=Y$, we write ${\mathcal L}(X)$, $S_p(X)$,
     and $S_{p,{\rm w}}(X)$, instead of ${\mathcal L}(X,X)$, $S_p(X,X)$, and $S_{p,{\rm w}}(X,X)$, respectively. Moreover, whenever appropriate, we omit $X$ and $Y$ in the notations ${\mathcal L}$, $S_p$, and $S_{p,{\rm w}}$.\\ If $p \geq 1$, then $\|\cdot\|_p$ is a norm, and $S_p$ is a Banach space. If $p > 1$, then there exists a norm in $S_{p,{\rm w}}$ which is equivalent to the functional $\|\cdot\|_{p, {\rm w}}$, and $S_{p,{\rm w}}$, equipped with this norm, is again a Banach space. Moreover, evidently, if $0 < p_1 \leq p_2 < p_3$, then $S_{p_1} \subset S_{p_2, {\rm w}} \subset S_{p_3}$, and all the inclusions are strict.\\
     For further references, we introduce here the eigenvalue counting functions for compact operators. Let $T=T^* \in S_\infty$. For $s>0$ set
    \bel{22}
    n_\pm(s; T) : = {\rm Tr}\,\one_{(s,\infty)}(\pm T).
    \ee
    Here and in the sequel $\one_S$ denotes the characteristic function of the set $S$; thus $\one_{\mathcal I}(T)$ is the spectral projection of $T$ corresponding to the interval ${\mathcal I} \subset \re$, and $n_+(s;T)$ (resp., $n_-(s;T)$) is just the number of the eigenvalues of the operator $T$ larger than $s$ (resp., smaller than $-s$), counted with their multiplicities.
    If $T_j = T_j^* \in S_\infty(X)$, $j=1,2$,  then the Weyl inequalities
    \bel{23}
    n_{\pm}(s_1+s_2; T_1+T_2) \leq n_\pm(s_1; T_1) + n_\pm(s_2; T_2)
    \ee
    hold for $s_j > 0$, $j=1,2$, (see e.g. \cite[Theorem 9, Section 9.2]{bs2}). \\
    Let $T \in S_\infty(X,Y)$. For $s>0$ set
    \bel{24}
    n_*(s; T) : = n_+(s^2; T^*T).
    \ee
    Thus, $n_*(s;T)$ is the number of the singular values of the operator $T$, larger than $s$, and counted with their multiplicities.
 \subsection{Some known results}
Let us now turn to the study of the spectral properties of the harmonic Toeplitz operators $T_V = PV$. Assume at first that $V \in C(\overline{\Omega})$; then, evidently, $T_V$ is bounded.
    Our first proposition deals with the location of $\sigma_{\rm ess}(\ttv)$, and contains a criterion for the compactness of  $\ttv$.

    \begin{pr} \label{p1}
     Let $\Omega \subset \re^d$, $d \geq 2$, be a bounded domain with boundary $\partial \Omega \in C^\infty$. Let $V \in C(\overline{\Omega})$.\\
    {\rm (i)} {\rm \cite[Theorem 4.5]{cln}} We have $\sigma_{\rm ess}(\ttv) = V(\partial \Omega)$.\\
    {\rm (ii)} {\rm \cite[Corollary 4.7]{cln}} The operator $\ttv$ is compact in $\harmo$ if and only if $V = 0$ on $\partial \Omega$.
    \end{pr}

Further, it is well known that the projection $P$ onto ${\mathcal H}(\Omega)$ (see \eqref{1}) admits an integral kernel ${\mathcal R} \in C^\infty(\Omega \times \Omega)$, called {\em the reproducing kernel} of $P$ (see e.g. \cite{kako, cln}). Thus
$$
(Pu)(x) = \int_\Omega {\mathcal R}(x,y) u(y) dy, \quad x \in \Omega, \quad u \in L^2(\Omega).
$$
Let $\left\{\varphi_j\right\}_{j \in \N}$ be an orthogonal basis in ${\mathcal H}(\Omega)$. Then we have
    \bel{2}
    {\mathcal R}(x,y) = \sum_{j \in \N} \varphi_j(x) \overline{\varphi_j(y)}, \quad x,y \in \Omega,
    \ee
    the series being locally uniformly convergent in $\Omega \times \Omega$. Evidently, ${\mathcal R}(x,y)$ is independent of the choice of the basis  $\left\{\varphi_j\right\}_{j \in \N}$. Moreover, the kernel ${\mathcal R}$ is real-valued and symmetric.
    For $x \in \Omega$ put
    $$
\varrho(x): = {\mathcal R}(x,x).
    $$
    Then, \eqref{2} implies that
    $$
    |{\mathcal R}(x,y)| \leq \varrho(x)^{1/2}\,\varrho(y)^{1/2}, \quad x,y \in \Omega.
    $$
For $x, y \in \Omega$, set
    \bel{sep1}
r(x) : = {\rm dist}(x,\partial \Omega), \quad \delta(x,y) : = |x-y| + r(x) + r(y).
    \ee
    \begin{lemma} \label{l1}
    {\rm \cite[Theorem 1.1]{kako}} For any multiindices $\alpha, \beta \in \Z_+^d$ there exists a constant $C_{\alpha, \beta} \in (0,\infty)$ such that
    \bel{a}
    \left|D_x^\alpha D_y^\beta {\mathcal R}(x,y)\right| \leq \frac{C_{\alpha, \beta}}{\delta(x,y)^{d+|\alpha| + |\beta|}}, \quad x,y \in \Omega.
    \ee
    Moreover, there exists a constant $C \in (0,\infty)$ such that
    \bel{b}
    \varrho(x) \geq C r(x)^{-d}, \quad x \in \Omega.
    \ee
    \end{lemma}
    For a Borel set ${\mathcal A} \subset \Omega$ set $\rho({\mathcal A}) := \int_{\mathcal A} \varrho(x)dx$. By \eqref{a} with $\alpha = \beta = 0$, and  \eqref{b}, $\rho$ is an infinite $\sigma$-finite measure on $\Omega$ which is absolutely continuous with respect to the Lebesgue measure.\\

    The following proposition contains criteria for the boundedness, compactness and membership to $S_p$, $p \in [1,\infty)$, of $T_V$ in the case where $0 \leq V \in L^1(\Omega)$.
    In fact, following \cite{cln}, we will formulate these results in a more general setting, considering harmonic Toeplitz operators $T_\mu$ associated with  finite Borel measures $\mu \geq 0$ on $\Omega$. In this case, $T_\mu$ is defined by
    $$
    (T_\mu u)(x) : = \int_\Omega {\mathcal R}(x,y) u(y) d\mu(y), \quad u \in \harmo, \quad x \in \Omega.
    $$
    If $d\mu(x) = V(x)dx$ with $0 \leq V \in L^1(\Omega)$, then, of course, $T_\mu = T_V$. Define {\rm the Berezin transform} $\tilde{\mu}$ of the measure $\mu$ by
    \bel{s2}
    \tilde{\mu}(x) : = \varrho(x)^{-1} \int_\Omega {\mathcal R}(x,y)^2 d\mu(y), \quad x \in \Omega.
    \ee
    In what follows we write $A \asymp B$ if there exist constants $0 < c_1 \leq c_2 < \infty$ such that $c_1 A \leq B  \leq c_2 A$. \\

    \begin{pr} \label{p2} Let $\Omega \subset \re^d$, $d \geq 2$, be a bounded domain with boundary $\partial \Omega \in C^\infty$.
    Let $\mu \geq 0$ be a finite Borel measure on $\Omega$, and let $\tilde{\mu}$ be its Berezin transform. \\
    {\rm (i) \cite[Theorem 3.5, Theorem 3.9]{cln}} We have $T_\mu \in {\mathcal L}(\harmo )$ if and only if $\tilde{\mu}$ is bounded on $\Omega$.
    Moreover,
    \bel{4}
    \|T_\mu\| \asymp \sup_{x \in \Omega}\tilde{\mu}(x).
    \ee
     {\rm (ii) \cite[Theorem 3.11, Theorem 3.12]{cln}} We have $T_\mu \in S_\infty(\harmo)$ if and only if $$\lim_{x \to \partial \Omega}\tilde{\mu}(x) = 0.$$\\
     {\rm (iii) \cite[Theorem 3.13]{cln}} Let $p \in [1,\infty)$. We have $T_\mu \in S_p(\harmo)$ if and only if $\tilde{\mu} \in L^p(\Omega;d\rho)$.
    Moreover,
    \bel{5}
    \|T_\mu\|_p \asymp \|\tilde{\mu}\|_{L^p(\Omega;d\rho)}.
    \ee
    \end{pr}
  \subsection{Membership to weak Schatten-von Neumann classes}  Our next goal is to establish conditions which guarantee $T_V \in S_{p, {\rm w}}(\harmo)$, $p \in (1,\infty)$. As a by-product we obtain also simple-looking sufficient conditions which imply $T_V \in S_{p}(\harmo)$, $p \in [1,\infty)$. \\
    For $p \in (0,\infty)$ define $L_{\rm w}^p(\Omega;d\rho)$ as the class of $\rho$-measurable functions $u: \Omega \to \C$ for which the quasinorm
    $$
    \|u\|_{L^p_{\rm w}(\Omega;d\rho)} : = \sup_{t>0} t \rho\left(\left\{x \in \Omega \, | \, |u(x)| > t\right\}\right)^{1/p}
    $$
    is finite. If $p>1$, then there exists a norm in $L_{\rm w}^p(\Omega;d\rho)$ which is equivalent to the functional $\|\cdot\|_{L^p_{\rm w}(\Omega;d\rho)}$, and $L_{\rm w}^p(\Omega;d\rho)$, equipped with this norm, is a Banach space.
    \begin{pr} \label{p3}
    Let $\Omega \subset \re^d$, $d \geq 2$ be a bounded domain with boundary $\partial \Omega \in C^\infty$. \\
    {\rm (i)} Assume $V \in L^p(\Omega;d\rho)$, $p \in [1,\infty)$. Then $T_V \in S_p(\harmo)$ and
    \bel{6}
    \|T_V\|_p \leq \|V\|_{L^p(\Omega;d\rho)}.
    \ee
    {\rm (ii)} Assume $V \in L^p_{\rm w}(\Omega;d\rho)$, $p \in (1,\infty)$. Then $T_V \in S_{p, {\rm w}}(\harmo)$ and
    \bel{7}
    \|T_V\|_{p, {\rm w}} \leq \|V\|_{L^p_{\rm w}(\Omega;d\rho)}.
    \ee
    \end{pr}
    \begin{proof} Let us consider the operator $PVP$ as defined on $L^2(\Omega)$. Evidently,
    \bel{s50}
    \|T_V\|_p = \|PVP\|_p, \quad \|T_V\|_{p, \rm{w}} = \|PVP\|_{p, \rm{w}}, \quad p \in (0,\infty).
    \ee
    We have $PVP = F^* e^{i\,{\rm arg}\,V} F$ where $F: L^2(\Omega) \to L^2(\Omega)$ is the operator with integral kernel
    $$
    |V(x)|^{1/2} {\mathcal R}(x,y), \quad x,y \in \Omega.
    $$
     Assume  $V \in L^1(\Omega;d\rho)$. Then
    \bel{8}
    \|PVP\|_1 \leq \|F^*\|_2 \|e^{i\,{\rm arg}\,V}\| \|F\|_2 = \|F\|_2^2 = \|V\|_{L^1(\Omega;d\rho)}.
    \ee
    Assume now  $V \in L^\infty(\Omega;d\rho)$. Since $\|P\| =1$ and $d\rho$ is absolutely continuous with respect to the Lebesgue measure,
    \bel{9}
    \|PVP\| \leq \|V\|_{L^\infty(\Omega)} = \|V\|_{L^\infty(\Omega;d\rho)}.
    \ee
     Interpolating between \eqref{8} and \eqref{9}, and applying \cite[Theorem 3.1]{bs1}, we find that
     $$
     \|PVP\|_p \leq \|V\|_{L^p(\Omega;d\rho)}, \quad p \in [1,\infty),
     $$
     $$
     \|PVP\|_{p, {\rm w}} \leq \|V\|_{L^p_{\rm w}(\Omega;d\rho)}, \quad p \in (1,\infty),
     $$
     which combined with \eqref{s50}, implies \eqref{6} and \eqref{7}.
     \end{proof}
     {\em Remark: }
     We believe that the main part of Proposition \ref{p3} is the second one, while the first part is just a by-product of the interpolation method applied, and is obviously less sharp than Proposition \ref{p2} (iii). Let us still point out some of the aspects of estimates \eqref{6} which we consider valuable:
     \begin{itemize}
     \item The estimating constant in \eqref{6} is just equal to one while the constants in \eqref{5} are not explicit and may depend on $\Omega$.
     \item The boundedness of $\Omega$ in Proposition \ref{p2} is essential, while estimates \eqref{6}  remain valid for generic unbounded domains.
     \item Estimates \eqref{6} are given in terms of $V$ itself, while estimates \eqref{5} are given in terms of its Berezin transform.
     \end{itemize}
\subsection{Berezin theory's point of view} Let us recall briefly the Berezin theory of operators with covariant and contravariant symbols (see \cite{ber} or \cite[Section 2, Chapter V]{bershu}). Let $X$ be a separable Hilbert space with scalar product $\langle \cdot , \cdot\rangle_X$ and norm $\|\cdot\|_X$, and let $M$ be a space with measure $\lambda$. Introduce the family $\left\{\eeps_m\right\}_{m \in M}$ such that $\|\eeps_m\|_X = 1$, $m \in M$, and for any $f \in X$ the function $M \ni m \mapsto \langle f,\eeps_m \rangle$ is measurable, and we have
     $$
     \|f\|_X^2 = \int_M |\langle f,\eeps_m \rangle |^2d\lambda(m).
     $$
     Further, define the orthogonal projection $P_m : = \langle\cdot,\eeps_m\rangle_X \, \eeps_m$, $m \in M$. Assume that $a \in L^\infty(M;d\lambda)$ and define the operator
     $$
     T : = \int_M a(m) P_m d\lambda(m),
     $$
     the integral being understood in the weak sense. Finally, set
     $$
     b(m) : = \langle T\eeps_m, \eeps_m\rangle_X, \quad m \in M.
     $$
     Then $a$ is called {\em the contravariant symbol} of the operator $T$, while $b$ is called its {\em covariant symbol}. It is easy to check that we have
     \bel{s1a}
     \sup_{m \in M} |b(m)| \leq \|T\| \leq \|a\|_{L^\infty(M;d\lambda)}.
     \ee
     The harmonic Toeplitz operator $T_V$ fits well in this scheme if we choose
     $$
     X = {\mathcal H}(\Omega), \quad M = \Omega, \quad \lambda = \rho, \quad \eeps_m(x) = \varrho(m)^{-1/2} {\mathcal R}(m,x), \quad m,x \in \Omega.
     $$
     Then $V$ is the contravariant symbol of $T_V$ while its Berezin transform
     $$
     \tilde{V}(m) : = \varrho(m)^{-1}\,\int_M {\mathcal R}(m,y)^2 V(y) dy, \quad m \in \Omega,
     $$
     defined by analogy with \eqref{s2}, is the covariant symbol of $T_V$. From this point of view, if $d\mu = Vdx$ with $V \geq 0$, then the lower bound in \eqref{4} is equivalent to the first inequality in \eqref{s1a}, while \eqref{9} coincides with the second inequality in \eqref{s1a}. Proposition \ref{p2} (i) shows that, generally speaking, the estimates of $\|T_V\|$ in terms of $\tilde{V}$ are sharper than those in terms of $V$. On the other hand, if $V \geq 0$, then we have
     $$
     \|T_V\|_1 = {\rm Tr}\,T_V = \int_\Omega V(m) d\rho(m) = \int_\Omega \tilde{V}(m) d\rho(m).
     $$
     Thus, estimates \eqref{6} - \eqref{7} are obtained by interpolation between the sharp estimate \eqref{8} and the unsharp, in the general case, estimate \eqref{9}.
     Note however that there exist situations where the estimates in terms of $V$ may yield results which are sharp in order (see below Theorem \ref{t1} and the remark after it). \\

       {\em Remark}: The Berezin-Toeplitz operators related to the Fock-Segal-Bargmann holomorphic subspace of $L^2(\rd)$, and their generalizations corresponding to higher Landau levels, are known to play an important role in the spectral and scattering theory of quantum Hamiltonians in constant magnetic fields
    (see e.g. \cite{r0, rw, fr, bbr1, prvb, bbr2, lr}). In particular, Proposition 3.6 of \cite{lr} is an analogue of our Proposition \ref{p3} for such operators (see also
    \cite[Lemma 5.1]{r0} and \cite[Lemma 3.1]{fr} where however no weak Schatten-von Neumann classes were considered).
\subsection{Compactly supported symbols}
    Finally, we establish a result which shows that if the symbol $V$ is compactly supported in $\Omega$, then $T_V \in S_p$ for any $p \in (0,\infty)$, i.e. the singular numbers of $T_V$ decay very rapidly, even if the behaviour of $V$ is quite irregular. In fact, we will replace in this case $V$ by $\phi \in {\mathcal E}'(\Omega)$, the class of distributions over ${\mathcal E}(\Omega) : = C^\infty(\Omega)$. We recall that $\phi \in {\mathcal D}'(\Omega)$, the class of distributions over ${\mathcal D}(\Omega) : = C_0^\infty(\Omega)$, is in ${\mathcal E}'(\Omega)$, if and only if ${\rm supp}\, \phi$ is compact in $\Omega$. If $\phi \in {\mathcal E}'(\Omega)$, we define $T_\phi : \harmo \to \harmo$ as the operator with integral kernel
    $$
    K_\phi(x,y) : = \left(\phi, {\mathcal R}(x,\cdot) {\mathcal R}(\cdot,y)\right), \quad x,y \in \Omega,
    $$
    where $(\cdot,\cdot)$ denotes the pairing between  ${\mathcal E}'(\Omega)$ and ${\mathcal E}(\Omega)$. Of course, if $\phi = \mu$ and $\mu \geq 0$ is a finite Borel measure such that ${\rm supp}\, \mu$ is compact in $\Omega$, then $T_\phi = T_\mu$. \\
    Since ${\rm supp}\,\phi$ is compact in $\Omega$, we have $K_\phi \in C^\infty(\overline{\Omega} \times \overline{\Omega})$. Therefore,
    $$
    s_j(T_\phi) = O(j^{-m}), \quad \forall \, m \in (0,\infty),
    $$
    (see e.g. \cite[Proposition 2.1]{bs1}). Thus, we arrive at
    \begin{pr} \label{p4}
    Let $\Omega \subset \re^d$, $d \geq 2$, be a bounded domain with boundary $\partial \Omega \in C^\infty$.
    Assume that $\phi \in {\mathcal E}'(\Omega)$. Then we have $T_\phi \in S_p(\harmo)$ for any $p \in (0,\infty)$, and, hence,
    \bel{9a}
    n_*(\lambda; T_\phi) = O(\lambda^{-\alpha}), \quad \lambda \downarrow 0,
    \ee
    for any $\alpha \in (0,\infty)$.
    \end{pr}
     {\em Remarks}: (i) In Section \ref{symmetry} we will show that if $\Omega$ is the unit ball in $\re^d$, and $V \geq 0$ is compactly supported and ${\rm supp}\,V$ is radially symmetric, then the eigenvalues of $T_V$ decay exponentially fast. Hopefully, in a future work we will extend these results to more general domains, and more general compactly supported $V$.\\
     (ii) Harmonic Toeplitz operators $T_\phi$ with $\phi \in {\mathcal E}'(\Omega)$ were considered in \cite{ar} where, in particular, it was proved that ${\rm rank}\,T_\phi < \infty$, if and only if $\supp \,\phi$ is finite.

\section{Spectral asymptotics of $T_V$ for general $V$ of power-like decay at the boundary}
\label{ss23}
\subsection{Statement of the main results}
\label{sso1}
In this section we assume that $V : \overline{\Omega} \to [0,\infty)$ is sufficiently regular near $\partial \Omega$, and has a power-like decay at $\partial \Omega$. We investigate the asymptotic behaviour of the discrete spectrum of $T_V$ near the origin. We obtain the main asymptotic term of $n_+(\lambda; T_V)$ as $\lambda \downarrow 0$, and give a sharp estimate of the remainder (see Theorem \ref{t1} below).\\

For the statement of Theorem \ref{t1} we need the following notations. We consider $\partial \Omega$ as a compact $(d-1)$-dimensional Riemannian manifold with metric tensor $g(y) : = \left\{g_{jk}(y)\right\}_{j,k=1}^{d-1}$, $y \in \partial \Omega$, generated by the Euclidean metrics in $\re^d$. For $y \in \partial \Omega$ and $\eta \in T^*_y \partial \Omega = \re^{d-1}$ we set
$$
|\eta| = |\eta|_y : = \left(\sum_{j,k=1}^{d-1} g^{jk}(y) \eta_j \eta_k\right)^{1/2},
$$
where $\left\{g^{jk}(y)\right\}_{j,k=1}^{d-1}$ is the matrix inverse to  $g(y)$. Let $dS(y)$ be the measure induced by $g$ on $\partial \Omega$. As usually, we denote by $L^2(\partial \Omega)$ the Hilbert space $L^2(\partial \Omega; dS(y))$.

    Let $a, \tau \in C^\infty(\overline{\Omega})$ satisfy $a > 0$ on $\overline{\Omega}$,  $\tau > 0$ on $\Omega$, and $\tau = r:={\rm dist}(\cdot,\partial \Omega)$ (see \eqref{sep1}) in a vicinity of $\partial \Omega$. Assume that
    \bel{rez1}
    V(x) = \tau(x)^\gamma a(x),  \quad \gamma \geq 0, \quad x \in \Omega.
    \ee
    Set $a_0 : = a_{|\partial \Omega}$.

    \begin{theorem} \label{t1}
    Assume that $V$ satisfies \eqref{rez1} with $\gamma > 0$. Then we have
    \bel{a2}
    n_+(\lambda; T_V) = {\mathcal C} \, \lambda^{-\frac{d-1}{\gamma}}\left(1 + O(\lambda ^{\frac{1}{\gamma}})\right), \quad \lambda \downarrow 0,
    \ee
    where
    \bel{a3a}
    {\mathcal C}: = \omega_{d-1} \left(\frac{ \Gamma(\gamma + 1)^{\frac{1}{\gamma}}}{4\pi}\right)^{d-1} \int_{\partial \Omega} a_0(y)^{\frac{d-1}{\gamma}}\,dS(y),
    \ee
    $\omega_n = \pi^{n/2}/\Gamma(1+n/2)$ is the Lebesgue measure of the unit ball $B_1 \subset \re^n$, $n \geq 1$, and $\Gamma$ is the Euler gamma function.
    \end{theorem}

    {\em Remark}: Estimates \eqref{a} with $\alpha = \beta = 0$, and \eqref{b} imply that under the hypotheses of Theorem \ref{t1} we have $V \in L^p_{\rm w}(\Omega; d\rho)$ if and only if $p = \frac{d-1}{\gamma}$. Then, estimate \eqref{7} yields
    \bel{s6a}
    n_+(\lambda; T_V) \leq \lambda^{-\frac{d-1}{\gamma}} \|V\|_{L^{(d-1)/\gamma}_{\rm w}(\Omega;d\rho)}^{(d-1)/\gamma}, \quad \lambda>0,
    \ee
    if $d-1 > \gamma$.
    By \eqref{a2}, we find that under the hypotheses of Theorem \ref{t1}, estimate \eqref{s6a} is sharp in order, provided that $d-1 > \gamma$.\\

    The proof of Theorem \ref{t1} can be found in the next subsection, while Subsection \ref{sso3} contains some extensions of this theorem. \\

    \subsection{Proof of Theorem 3.1.}
    \label{sso2}
    For $s \in \re$ denote by $H^s(\Omega)$ and $H^s(\partial \Omega)$ the Sobolev spaces on $\Omega$ and $\partial \Omega$ respectively. Assume that
    $f \in H^s(\partial \Omega)$, $s \in \re$. Then the boundary-value problem
    \bel{a3}
    \left\{
    \begin{array} {l}
    \Delta u = 0 \quad \mbox{in} \quad \Omega,\\
     u = f \quad \mbox{on} \quad \partial \Omega,
     \end{array}
     \right.
     \ee
     admits a unique solution $u \in H^{s+1/2}(\Omega)$, we have
     \bel{a4}
     \|u\|_{H^{s+1/2}(\Omega)} \asymp  \|f\|_{H^{s}(\partial \Omega)},
     \ee
     and, therefore, the mapping $f \mapsto u$ defines an isomorphism between $H^{s}(\partial \Omega)$ and $H^{s+1/2}(\Omega)$
     (see \cite[ Sections 5, 6, 7, Chapter 2]{lm}). \\

     If $s=0$, we set
     \bel{a5}
     u = Gf.
     \ee
     By \eqref{a4} with $s=0$, and the compactness of the embedding of $H^{1/2}(\Omega)$ into $L^2(\Omega)$, we find that the operator $G: L^2(\partial \Omega) \to L^2(\Omega)$ is compact. By \cite[Theorem 12, Section 2.2]{evans}, we have
     \bel{a5a}
     u(x) = \int_{\partial \Omega} {\mathcal K}(x,y)  f(y) dS(y), \quad x \in \Omega,
     \ee
     where
     \bel{a11}
     {\mathcal K}(x,y) : = -\frac{\partial \mathcal G}{\partial \nu_y}(x,y), \quad x \in \Omega, \quad y \in \partial \Omega,
     \ee
     ${\mathcal G}$ is the Dirichlet Green function associated with $\Omega$, and $\nu$ is the unit outer normal vector at $\partial \Omega$. Note that
     \bel{a5b}
     {\mathcal K} \in C^\infty(\Omega \times \partial \Omega).
     \ee
     \begin{lemma} \label{l2}
     We have
     \bel{a9}
     {\rm Ker}\,G = \{0\},
     \ee
     \bel{a6}
     \overline{{\rm Ran}\,G} = {\mathcal H}(\Omega).
     \ee
     \end{lemma}
     \begin{proof}
     Relation \eqref{a9} follows from \eqref{a4} with $s=0$. Let us check \eqref{a6}.
      Pick $u \in  {\mathcal H}(\Omega)$. Then, by \eqref{a3} with $s=-1/2$, we have $f : = u_{|\partial \Omega} \in H^{-1/2}(\partial \Omega)$. Let $f_n \in L^2(\partial \Omega)$, $n \in \N$, and
     \bel{a7}
     \lim_{n \to \infty} \|f_n - f\|_{H^{-1/2}(\partial \Omega)} = 0.
     \ee
     Set $u_n : = G f_n$. Then $u_n \in {\rm Ran}\,G$, $n \in \N$, and by \eqref{a4} with $s=-1/2$, and \eqref{a7}, we have $\lim_{n \to \infty} \|u_n - u\|_{L^{2}(\Omega)} =0$ which implies \eqref{a6}.
     \end{proof}
     Set $J : = G^* G$. Then the operator $J = J^* \geq 0$ is compact in $L^2(\partial \Omega)$. Due to  \eqref{a9}, we have ${\rm Ker}\,J = \{0\}$. Let $\left\{\lambda_j\right\}_{j \in \N}$ be the non-increasing sequence of the eigenvalues $\lambda_j>0$ of $J$, and let
     $\left\{\phi_j\right\}_{j \in \N}$ be the corresponding orthonormal eigenbasis in $L^2(\partial \Omega)$ with $J \phi_j = \lambda_j \phi_j$, $j \in \N$. Define the operator $J^{-1}$, self-adjoint in $L^2(\partial \Omega)$, by
     \bel{sep3}
     J^{-1}u : = \sum_{j \in \N} \lambda_j^{-1} \langle u, \phi_j\rangle \phi_j, \quad
      {\rm Dom}\,J^{-1} : = \left\{u \in L^2(\partial \Omega) \, | \, \sum_{j \in \N} \lambda_j^{-2} \left|\langle u, \phi_j\rangle\right|^2 < \infty\right\},
     \ee
     $\langle \cdot,\cdot\rangle$ being the scalar product in $L^2(\partial \Omega)$. Evidently, $\overline{J J^{-1}} = J^{-1} J = I$. \\

     Further, write the polar decomposition of the operator $G = U |G| = U J^{1/2}$ where $U : L^2(\partial \Omega) \to L^2(\Omega)$ is an isometric operator. By Lemma \ref{l2}, we have ${\rm Ker}\,U = \{0\}$ and ${\rm Ran}\,U = {\mathcal H}(\Omega)$. Thus, we obtain the following
     \begin{pr} \label{p7}
     The orthogonal projection $P$ onto ${\mathcal H}(\Omega)$ satisfies
     \bel{a8}
     P = GJ^{-1}G^* = U U^*.
     \ee
     \end{pr}
     Assume that $V$ satisfies \eqref{rez1} with $\gamma \geq 0$, and set $J_V : = G^*VG$; from this point of view, we have $J = J_1$.
     \begin{pr} \label{p8}
     Let $V$ satisfy \eqref{rez1} with $\gamma \geq 0$. Then the operator $\ttv$  is unitarily equivalent to (the closure of) the operator $J^{-1/2} J_V J^{-1/2}$.
     \end{pr}
     \begin{proof}
     By \eqref{a8}, we have
     $$
     PVP = U J^{-1/2} G^* V G J ^{-1/2} U^* = U J^{-1/2} J_V J^{-1/2} U^*,
     $$
     and the operator $U$ maps unitarily $L^2(\partial \Omega)$ onto ${\mathcal H}(\Omega)$.
     \end{proof}
    \begin{pr} \label{p9}
    Under the assumptions of Proposition \ref{p8} the operator $J^{-1/2} J_V J^{-1/2}$ is a $\Psi${\rm DO}  with principal symbol
    \bel{a1}
    2^{-\gamma} \Gamma(\gamma + 1) |\eta|^{-\gamma} a_0(y), \quad (y,\eta) \in T^*\partial \Omega.
    \ee
    \end{pr}
    \begin{proof} Using the  pseudo-differential calculus due to L. Boutet de Monvel (see  \cite{bdm, bdmg}), M. Engli\v{s} showed recently in \cite[Sections 6, 7]{e} that if $V$ satisfies  \eqref{rez1} with $\gamma \geq 0$, then the operator $J_V$ is a $\Psi$DO with principal symbol
    $$
    2^{-\gamma-1} \Gamma(\gamma + 1) |\eta|^{-\gamma-1} a_0(y), \quad  (y,\eta) \in T^*\partial \Omega.
    $$
    In particular, $J = J_1$ is a $\Psi$DO with principal symbol $2^{-1} |\eta|^{-1}$. Then the pseudo-differential calculus (see e.g. \cite[Chapters I, II]{shu}) easily implies that $J^{-1/2}$ is a $\Psi$DO with principal symbol $2^{1/2} |\eta|^{1/2}$, and $J^{-1/2} J_V J^{-1/2}$ is a $\Psi$DO   with principal symbol defined in \eqref{a1}.
    \end{proof}

    Now we are in position to prove Theorem \ref{t1}. It is easy to see that under its assumptions we have ${\rm Ker}\,J^{-1/2} J_V J^{-1/2} = \{0\}$. Using the spectral theorem, define the operator
    $$
      A : = \left(J^{-1/2} J_V J^{-1/2}\right)^{-1/\gamma}
    $$
    (cf. \eqref{sep3}). Then, by the pseudo-differential calculus, $A$ is a  $\Psi$DO with principal symbol
    $$
    2 \Gamma(\gamma + 1)^{-1/\gamma} |\eta| a_0(y)^{-1/\gamma}, \quad (y,\eta) \in T^*\partial \Omega.
    $$
    By Proposition \ref{p8} and the spectral theorem, we have
    \bel{s6}
    n_+(\lambda; \ttv) = n_+(\lambda; J^{-1/2} J_V J^{-1/2}) = {\rm Tr}\,\one_{\left(-\infty, \lambda^{-1/\gamma}\right)}(A), \quad \lambda > 0.
    \ee
    A classical result of L. H\"ormander \cite{horman} easily implies that
    \bel{sep5}
    {\rm Tr}\,\one_{\left(-\infty, E\right)}(A) = {\mathcal C} E^{d-1}(1 + O(E^{-1})), \quad E \to \infty,
    \ee
    the constant ${\mathcal C}$ being defined in \eqref{a3a}. Combining \eqref{s6} and \eqref{sep5}, we arrive at \eqref{a2}.\\

    {\em Remark}: The natural idea to parametrize the functions $u \in {\mathcal H}(\Omega)$ by their restrictions on $\partial \Omega$ has been used in the theory of harmonic Toeplitz operators and  related areas by various authors; it could be traced back at least to the classical work \cite{bdm}, and has been recently applied in \cite{e} in order to obtain a suitable representation of the operator $J_V$. We would like to mention as well the article \cite{bs3} where the authors consider the operator generated by the ratio of two quadratic differential forms defined on the solutions of a homogeneous elliptic equation. The order of the numerator is lower than the order of the denominator, and, since the domain considered is supposed to be bounded and to have a regular boundary, the operator generated by the ratio is compact. \\
    The harmonic Toeplitz operator $T_V$ could be interpreted as the operator generated by the quadratic-form ratio
    \bel{bx1}
    \frac{\int_\Omega V |u|^2\,dx}{\int_\Omega |u|^2\,dx}, \quad u \in  {\mathcal H}(\Omega).
    \ee
    Note that both the numerator and the denominator in \eqref{bx1} are of zeroth order, and the compactness of $T_V$ is now due to the fact that $V$ vanishes at $\partial \Omega$. \\
    In spite of the differences between the operators considered in \cite{bs3}, and the harmonic Toeplitz operators studied here, the unitary equivalence of $T_V$ and $J^{-1/2} J_V J^{-1/2}$ established in our Proposition \ref{p8} has much in common with the reduction to a $\Psi$DO on $\partial \Omega$, performed in \cite{bs3}.\\

\subsection{Extensions of Theorem 3.1.}
    \label{sso3}
    In Theorem \ref{t1}, we assumed that $V$ was positive and smooth inside $\Omega$. In this section, we show that the result remains valid for more general $V$ which satisfy \eqref{rez1} only near $\partial \Omega$.

    \begin{follow} \label{fo1}
     Let $V$ satisfy the assumptions of Theorem \ref{t1}, and $\phi \in {\mathcal E}'(\Omega; \re)$.
     Then we have
      \bel{o4}
    n_+(\lambda; T_{V+\phi}) = {\mathcal C} \, \lambda^{-\frac{d-1}{\gamma}}\left(1 + O(\lambda ^{\frac{\eps}{\gamma}})\right), \quad \lambda \downarrow 0,
    \ee
     where $T_{V + \phi} : = T_V + T_\phi$, ${\mathcal C}$ is the constant defined in \eqref{a3a}, $\eps = 1$ if $d \geq 3$, and   $\eps < 1$ is arbitrary if $d = 2$.
     \end{follow}
     \begin{proof}
    The Weyl inequalities \eqref{23} imply
     $$
    n_+(\lambda(1+\lambda^\theta); T_V) - n_-(\lambda^{1+\theta}; T_\phi) \leq
    $$
    $$
     n_+(\lambda; T_{V+\phi}) \leq
     $$
     \bel{o1}
     n_+(\lambda(1-\lambda^\theta); T_V) + n_+(\lambda^{1+\theta}; T_\phi),
     \ee
     for $\lambda \in (0,1)$ and $\theta > 0$. By \eqref{a2},
     $$
     n_+(\lambda(1\pm\lambda^\theta); T_V) =
     $$
     \bel{o2}
     {\mathcal C}\left(\lambda(1\pm\lambda^\theta)\right)^{-\frac{d-1}{\gamma}} + O\left(\lambda^{-\frac{d-2}{\gamma}}\right) =
     {\mathcal C} \lambda^{-\frac{d-1}{\gamma}} + O\left(\lambda^{-\frac{d-2}{\gamma}}\right), \quad \lambda \in (0,1),
     \ee
     provided that $\theta > 1/\gamma$. Next, by estimate \eqref{9a}, we have
     \bel{o3}
     n_{\pm}(\lambda^{1+\theta}; T_\phi) = O(\lambda^{-\alpha(1+\theta)}), \quad \lambda > 0,
     \ee
     for any $\alpha \in (0,\infty)$. Assume $d \geq 3$ and choose $\alpha \in \left(0, \frac{d-2}{\gamma(1+\theta)}\right)$. Then \eqref{o4} follows from
     \eqref{o1} - \eqref{o3}. If $d=2$, then we can pick any $\varepsilon < 1$ and choose $\alpha \in \left(0, \frac{1-\varepsilon}{\gamma(1+\theta)}\right)$, in order to check that in this case \eqref{o1} -- \eqref{o3} again imply \eqref{o4}.
     \end{proof}

    {\em Remarks}: (i) Corollary \ref{fo1} implies that if $d \geq 3$, then Theorem \ref{t1} remains true if we replace $T_V$ by $T_{V+\phi}$ with $\phi \in {\mathcal E}'(\Omega; \re)$. In particular, it is valid also for potentials $V \in L^1_{\rm loc}(\Omega; \re)$ which satisfy \eqref{rez1} only in a neighborhood of $\partial \Omega$.\\
    (ii) Arguing as in the proof of Theorem \ref{t1} (see Propositions \ref{p8} and \ref{p9}), we can show that $T_{V+\phi}$ with $\phi \in {\mathcal E}'(\Omega; \re)$ is unitarily equivalent to self-adjoint $\Psi$DO with principal symbol defined in \eqref{a1}. The only problem to extend in a straightforward manner our proof of Theorem \ref{t1} to $T_{V+\phi}$ is that this operator may have a non trivial kernel unless, for example, $\phi \geq 0$. In particular, if $d=2$ and $\phi \in {\mathcal E}'(\Omega; \re)$ satisfies $\phi \geq 0$, then \eqref{o4} holds also for $\eps = 1$.

\section{Spectral properties of compactly supported $T_V$}
\label{symmetry}
In this section we assume that $\Omega = B_1$ where
$$
B_R : = \left\{x \in \re^d \, | \, |x| < R\right\}, \quad d \geq 2, \quad R \in (0,\infty).
$$
Thus, $\partial \Omega = {\mathbb S}^{d-1} : = \left\{x \in \re^d \, | \, |x| = 1\right\}$.
The space ${\mathcal H}(B_1)$ admits an explicit orthonormal eigenbasis which we are now going to describe.
 Recall that $k(k+d-2)$, $k \in {\mathbb Z}_+$, are the eigenvalues of the Beltrami-Laplace operator $-\Delta_{{\mathbb S}^{d-1}}$, self-adjoint in $L^2({\mathbb S}^{d-1})$ (see e.g. \cite[Section 22]{shu}). Moreover,
$$
{\rm dim}\, {\rm Ker}\,\left(-\Delta_{{\mathbb S}^{d-1}} - k(k+d-2) I\right) = : m_k =
\binom{d+k-1}{d-1} - \binom{d+k-3}{d-1}
$$
where  $\binom{m}{n} = \frac{m!}{(m-n)! \, n!}$ if $m \geq n$, and $\binom{m}{n} = 0$ if $m < n$ (see e.g. \cite[Theorem 22.1]{shu}).
Set
$$
    M_k : = \binom{d+k-1}{d-1} + \binom{d+k-2}{d-1}, \quad k \in \Z.
    $$
    Evidently,
    \bel{10}
    M_k = \frac{2k^{d-1}}{(d-1)!} \left(1 + O\left(k^{-1}\right)\right), \quad k \to \infty,
    \ee
    (see e.g. \cite[Eq. 6.1.47]{abst}).
By induction, we easily find that
    \bel{26}
\sum_{j=0}^k m_j = M_k, \quad k \in \Z_+.
    \ee
    Let $\psi_{k,\ell}$, $\ell =1,\ldots,m_k$, be an orthonormal basis in ${\rm Ker}\,\left(-\Delta_{{\mathbb S}^{d-1}}- k(k+d-2) I\right)$, $k \in {\mathbb Z}_+$. It is well known that $\psi_{k,\ell}$ are restrictions on ${\mathbb S}^{d-1}$ of homogeneous polynomials of degree  $k$, harmonic in $\re^d$ (see e.g \cite[Section 22]{shu}). Then the functions $\phi_{k,\ell}(x) : = \sqrt{2k+d} \, |x|^k \psi_{k,\ell}(x/|x|)$,
$x \in B_1$, $\ell = 1,\ldots,m_k$, $k \in {\mathbb Z}_+$, form an orthonormal basis in ${\mathcal H}(B_1)$. Let  ${\mathcal H}_k(B_1)$, $k \in {\mathbb Z}_+$, be the subspace of  ${\mathcal H}(B_1)$ generated by  $\phi_{k,\ell}$, $\ell = 1,\ldots,m_k$.  \\
Further, let $V(x) = v(|x|)$, $x \in B_1$, and let $v : [0,1) \to \re$ satisfy $\lim_{r \uparrow 1} v(r) = 0$, $v \in L^1((0,1); r^{d-1}dr)$.
Then the operator $T_V$ is self-adjoint and compact in ${\mathcal H}(B_1)$, and
\bel{11}
   T_V u = \mu_k u, \quad u \in {\mathcal H}_k(B_1),
   \ee
   where
   \bel{12}
   \mu_k(v) : = (2k+d) \int_0^1 v(r) r^{2k+d-1} dr, \quad k \in {\mathbb Z}_+.
   \ee
   Set
   $$
   \nu_\pm(s; v) = \#\left\{k \in \Z_+ \, | \, \mu_k(\pm v) > s\right\}, \quad s>0.
   $$

   Let us calculate  the eigenvalues of $T_V$ in a simple model situation where, in particular, $v \geq 0$ so that $T_V \geq 0$. More precisely, let $v(r) = b \,\one_{[0,c]}(r)$, $r \in [0,1)$, with $b>0$, and $c \in (0,1)$. Then \eqref{12} implies
   \bel{14}
   \mu_k(v) = b \, c^{2k+d}, \quad k \in \Z_+.
   \ee
   Evidently, the sequence $\left\{\mu_k(v)\right\}_{k \in \Z_+}$ is decreasing. Setting $V(x) : = v(|x|)$, $x \in \re^d$, we get
   \bel{31}
   n_+(\lambda; T_{V})  = M_{\nu_+(\lambda; v)-1}, \quad \lambda > 0.
   \ee
   Let us discuss the asymptotics of $ n_+(\lambda; T_{V})$ as $\lambda \downarrow 0$. By \eqref{14},
   \bel{33}
   \nu_+(\lambda; v) = \frac{1}{2} \frac{|\ln{\lambda}|}{|\ln{c}|} + O(1), \quad \lambda \downarrow 0.
   \ee
   By \eqref{31}, \eqref{10}, and \eqref{33}, we get
   \bel{15}
   n_+(\lambda; T_{V})  = \frac{2^{-d+2}}{(d-1)!|\ln{c}|^{d-1}} |\ln{\lambda}|^{d-1} + O\left(|\ln{\lambda}|^{-d+2}\right),  \quad \lambda \downarrow 0.
   \ee

   {\em Remark}: The fact that the basis $\left\{\phi_{k,\ell}\right\}$ diagonalizes the operator $T_V$ with radially symmetric symbol $V$, acting in ${\mathcal H}(B_1)$, was noted in \cite[Part 2.3.2]{roz}, and was used there, in particular, to obtain asymptotic relations of type \eqref{15}. The fact that the Toeplitz operators with radially symmetric symbols, acting in the holomorphic {\em Fock-Segal-Bargmann} space, are diagonalized in a certain canonic basis, was utilized already in \cite{rw, gv}. A similar result concerning Toeplitz operators with radially symmetric symbols, acting in the holomorphic {\em Bergman} space, can be found in \cite{gkv}. \\

   Next, we use \eqref{15} to study the spectral asymptotics for Toeplitz operators with symbols $V$ which are compactly supported in $\Omega$, and possess partial radial symmetry.

   \begin{pr} \label{p6} Let $\Omega = B_1$. Assume that $V : B_1 \to [0,\infty)$ satisfies $V \in L^\infty(B_1)$ and ${\rm supp}\,V = \overline{B_c}$ for some $c \in (0,1)$. Suppose moreover that for any $\delta \in (0,c)$ we have ${\rm ess}\,{\rm inf}_{x \in B_\delta} V(x) > 0$. Then
   \bel{21}
   \lim_{{\lambda} \downarrow 0} |\ln{\lambda}|^{-d+1} \, n_+(\lambda; T_{V}) = \frac{2^{-d+2}}{(d-1)!|\ln{c}|^{d-1}}.
   \ee
   \end{pr}
   \begin{proof}
   Pick $\delta \in (0,c)$. Then for almost every $x \in B_1$ we have
   $$
   b_- \one_{B_\delta}(x) \leq V(x)  \leq b_+ \one_{B_c}(x),
   $$
   where
   $$
   b_- : = {\rm ess}\,{\rm inf}_{x \in B_\delta} V(x), \quad b_+ : = {\rm ess}\,{\rm sup}_{x \in B_1} V(x).
   $$
   Then the mini-max principle and \eqref{15}  imply
   $$
    \frac{2^{-d+2}}{(d-1)!|\ln{\delta}|^{d-1}} \leq
    $$
    $$
    \liminf_{{\lambda} \downarrow 0} |\ln{\lambda}|^{-d+1} \, n_+(\lambda; T_{V}) \leq
   \limsup_{{\lambda} \downarrow 0} |\ln{\lambda}|^{-d+1} \, n_+(\lambda; T_{V}) \leq
   $$
   $$
   \frac{2^{-d+2}}{(d-1)!|\ln{c}|^{d-1}}.
   $$
   Letting $\delta \uparrow c$, we obtain \eqref{21}.
   \end{proof}

{\em Remarks}: (i) We do not estimate the remainder in \eqref{21} due to the fairly general assumptions concerning the behaviour of $V$ on ${\rm supp}\,V$.\\
(ii) Evidently, Proposition \ref{p6} could be easily extended to more general radially symmetric supports of $V$ which may contain, say, spherical layers and a ball. Further, if $d=2$, the proposition could be extended to non radially symmetric $\Omega$ and ${\rm supp}\, V$, applying appropriate conformal mappings. Possibly, such an approach based on complex-analytic methods, may also work in  arbitrary {\em even} dimensions $d$. We omit these extensions, with the hope that we will be able to develop a general method to extend Proposition \ref{p6} which would work in any, even or odd, dimension $d$. \\
(ii) Let $v(r) = a(1-r)^{\gamma}$, $r \in [0,1)$, with $a>0$ and $\gamma > 0$. Then, by \eqref{12}, we have
   \bel{29}
   \mu_k(v) = a (2k+d) {\rm B}(\gamma+1,2k+d) = a \Gamma(\gamma + 1) \frac{\Gamma(2k+d+1)}{\Gamma(2k+d+1+\gamma)}, \quad k \in \Z_+,
   \ee
   where ${\rm B}$ is the Euler beta functions.
   It is easy to show that the sequence $\left\{\mu_k(v)\right\}_{k \in \Z_+}$ is again decreasing.
   Setting as above $V(x) : = v(|x|)$, $x \in \re^d$, we find that  \eqref{29} implies
    \bel{16}
     n_+(\lambda; T_{V}) = \frac{2^{-d+2}}{(d-1)!} \left(a\Gamma(\gamma+1)\right)^{(d-1)/\gamma}\lambda^{-(d-1)/\gamma}(1 + o(1)), \quad \lambda \downarrow 0.
      \ee
   Thus, if we assume that
 $0 \leq V \in L^1(B_1)$, and there exist $\gamma > 0$ and $a > 0$ such that $\lim_{|x| \uparrow 1} (1-|x|)^{-\gamma} V(x) = a$, uniformly with respect to $x/|x| \in {\mathbb S}^{d-1}$,  we have
   \bel{18}
   \lim_{\lambda \downarrow 0} \lambda^{(d-1)/\gamma} \, n_+(\lambda; T_{V}) = \frac{2^{-d+2}}{(d-1)!} \left(a\Gamma(\gamma+1)\right)^{(d-1)/\gamma}.
   \ee
   We omit the simple proof of \eqref{18}, based on \eqref{16}, \eqref{9a}, and standard variational techniques, since up to regularity issues and absence of a remainder estimate, asymptotic relation \eqref{18} is a special case of \eqref{a2}.

\section{Applications to the spectral theory of the perturbed Krein Laplacian}\label{s3}
In this section we introduce the Krein Laplacian $K$, perturb it by a multiplier $V \in C(\overline{\Omega};\re)$, and investigate the spectral properties of the perturbed operator $K + V$.\\

For $ s\in \re$, we denote, as usual, by $H_0^s(\Omega)$ the closure of $C_0^\infty(\Omega)$ in the topology of the Sobolev space $H^s(\Omega)$. Set also $H_D^2(\Omega) : = H^2(\Omega) \cap H^1_0(\Omega)$. Define the minimal Laplacian
$$
\Delta_{\rm min} : = \Delta, \quad {\rm Dom}\,\Delta_{\rm min} = H_0^2(\Omega).
$$
As is well known, $\Delta_{\rm min}$ is symmetric but not self-adjoint in $L^2(\Omega)$, since we have
    \bel{s20}
    \Delta^*_{\rm min}  = :\Delta_{\rm max} = \Delta, \quad {\rm Dom}\,\Delta_{\rm max}  = \left\{u \in L^2(\Omega) \, | \, \Delta\,u \in L^2(\Omega)\right\},
    \ee
    $\Delta u$ being the distributional Laplacian of $u \in L^2(\Omega)$. Note that we have
    $$
    {\rm Ker}\,\Delta_{\rm max} = \harmo.
    $$
    \begin{lemma} \label{l3}
     The domain ${\rm Dom}\,\Delta_{\rm max} $ admits the direct-sum decomposition
    \bel{s23}
    {\rm Dom}\,\Delta_{\rm max}  = \harmo \dotplus H_D^2(\Omega).
    \ee
    \end{lemma}
    \begin{proof}
    Let us first show that the sum at the r.h.s. of \eqref{s23} is direct. Assume that $u_1 \in \harmo$, $u_2 \in H_D^2(\Omega)$, and $u_1 + u_2 = 0$. Then $u_2 \in H^2(\Omega)$ satisfies the homogeneous boundary-value problem
    $$
    \left\{
    \begin{array} {l}
    \Delta u_2 = 0 \quad \rm{in} \quad \Omega,\\
    u_2 = 0 \quad \rm{on} \quad \partial \Omega.
    \end{array}
    \right.
    $$
    Hence, $u_2 = 0$, and $u_1 = 0$. Evidently, if $u_1 \in \harmo$, $u_2 \in H_D^2(\Omega)$, then $u_1 + u_2 \in {\rm Dom}\,\Delta_{\rm max} $.
    Pick now $u \in {\rm Dom}\,\Delta_{\rm max} $, and let us check the existence of $u_1$ and $u_2$ such that
    \bel{s24}
    u_1 \in \harmo, \quad u_2 \in H_D^2(\Omega), \quad u = u_1 + u_2.
    \ee
    Define the Dirichlet Laplacian
    $$
    \Delta_D : = \Delta, \quad {\rm Dom}\,\Delta_{D}  : = H_D^2(\Omega).
    $$
    Set
    $$
    u_2 : = \Delta_D^{-1} \Delta u, \quad u_1 : = u - u_2.
    $$
    Evidently, $u_1$ and $u_2$ satisfy \eqref{s24}.
    \end{proof}
    Introduce the  Krein Laplacian
    $$
    K : = - \Delta, \quad {\rm Dom}\,K  = \harmo \dotplus H_0^2(\Omega).
    $$
    The operator $K \geq 0$, self-adjoint in $L^2(\Omega)$, is the von Neumann - Krein ``soft" extension of $-\Delta_{\rm min}$, remarkable for the fact that any other self-adjoint extension
    $S \geq 0$ of  $-\Delta_{\rm min}$ satisfies $$(S+ I)^{-1} \leq (K+I)^{-1}$$ (see \cite{jvn, k1}).  Evidently, ${\rm Ker}\,K = \harmo$. The domain ${\rm Dom}\,K$ admits a more explicit description in the terms of the Dirichlet-to-Neumann operator ${\mathcal D}$. For $f \in C^\infty(\partial \Omega)$, ${\mathcal D}\,f$ is defined by
    $$
    {\mathcal D}\,f = \frac{\partial u}{\partial \nu}_{|\partial \Omega},
    $$
    where $u$ is the solution of the boundary-value problem
    $$
    \left\{
    \begin{array} {l}
    \Delta u = 0 \quad {\rm in} \quad \Omega, \\
    u = f \quad {\rm on} \quad \partial \Omega.
    \end{array}
    \right.
    $$
    The operator ${\mathcal D}$ is a first-order elliptic operator; by the elliptic regularity, it extends to a bounded operator form $H^s(\partial \Omega)$ into $H^{s-1}(\partial \Omega)$, $s \in \re$.
    Then we have
    $$
    {\rm Dom}\,K  = \left\{u \in {\rm Dom}\,\Delta_{\rm max}  \, \left| \, \frac{\partial u}{\partial \nu}_{|\partial \Omega} =
    {\mathcal D}\left(u_{|\partial \Omega}\right)\right.\right\}
    $$
    (see \cite[Theorem III.1.2]{g1}). The Krein Laplacian $K$ arises naturally in the so called {\em abstract buckling problem} (see e.g. \cite{g2, agmst}). \\

     Denote by $L$ the restriction of $K$ onto ${\rm Dom}\,K  \cap \harmo^\perp$ where $\harmo^\perp : = L^2(\Omega) \ominus \harmo$.
    Then, $L$ is self-adjoint in the Hilbert space $\harmo^\perp$.

    \begin{pr} \label{p10} {\rm \cite{k1}, \cite[Theorem 5.1]{as}} The spectrum of $L$ is purely discrete and positive, and, hence, $L^{-1} \in S_\infty(\harmo^\perp)$. As a consequence,   $\sigma_{\rm ess}(K) = \{0\}$,  and the zero is an isolated eigenvalue of $K$ of infinite multiplicity.
    \end{pr}
Let $V \in C(\overline{\Omega}; \re)$. Then the operator $K+V$ with domain ${\rm Dom}\,K$ is self-adjoint in $L^2(\Omega)$. In the sequel, we will investigate the spectral properties of $K+V$.\\

    {\em Remarks}:  (i) In many aspects, the assumption $V \in C(\overline{\Omega})$ is too restrictive, the operator $K + V$ could also be self-adjoint on ${\rm Dom}\,K$  for less regular potentials $V$. Moreover, the sum $K + V$ could be defined in the sense of quadratic forms. However, the description of an optimal class of singular $V$ for which the sum $K + V$ is well defined in the operator or form sense requires additional technical work which is left for a possible future article.\\
    (ii) It should be underlined here that the perturbations $K_V$ of the Krein Laplacian $K$ discussed in \cite{agmt} are of different nature than the perturbations $K + V$ considered here. Namely, the authors of \cite{agmt} assume that $V \geq 0$, define the maximal operator $K_{V, \rm{max}}$ as
    $$
    K_{V, \rm{max}} : = - \Delta + V, \quad {\rm Dom}\, K_{V, \rm{max}} : = {\rm Dom}\,\Delta_{\rm max} ,
    $$
    and set
    $$
    K_V : = - \Delta + V, \quad {\rm Dom}\,K_V : = {\rm Ker}\,K_{V, \rm{max}} \dotplus H^2_0(\Omega).
    $$
    Thus, if $V \neq 0$, then the operators $K_V$ and $K_0 = K$ are self-adjoint on different domains, while the operators $K + V$ introduced here are  self-adjoint on the same domain ${\rm Dom}\, K$. It is shown in \cite{agmt} that for any $0 \leq V \in L^\infty(\Omega)$ we have $K_V \geq 0$, $\sigma_{\rm ess}(K_V) = \{0\}$,  and the zero is an isolated eigenvalue of $K_V$ of infinite multiplicity. As we will see in what follows, the spectral properties of $K + V$ could be quite different.

    \begin{theorem} \label{t2}
    Let $V \in C(\overline{\Omega}; \re)$. Then we have
    \bel{s23a}
    \sigma_{\rm ess}(K+V) = V(\partial \Omega).
    \ee
    In particular, $ \sigma_{\rm ess}(K+V) = \{0\}$ if and only if $V_{|\partial \Omega} = 0$.
    \end{theorem}
    \begin{proof}
    First, we will show that
    \bel{s21}
    (K+V-i)^{-1} - (K+PVP-i)^{-1} \in S_\infty(L^2(\Omega)).
    \ee
    Set $Q : = I - P$. Then
    $$
    (K+V-i)^{-1} - (K+PVP-i)^{-1} =
    $$
    \bel{s21a}
    - (K+V-i)^{-1} (K-i) (K-i)^{-1} (QVQ + PVQ + QVP) (K-i)^{-1} (K-i)(K+PVP-i)^{-1}.
    \ee
    Evidently,
    \bel{s21b}
     (K+V-i)^{-1} (K-i), \, (K-i)(K+PVP-i)^{-1}, \, P, \, V \in {\mathcal L}(L^2(\Omega)).
     \ee
     Moreover, using the orthogonal decomposition $L^2(\Omega) = \harmo \oplus \harmo^\perp$, and bearing in  mind Proposition \ref{p10}, we find that
     \bel{s22}
     Q(K-i)^{-1}, \, (K-i)^{-1}Q \in S_\infty(L^2(\Omega)).
     \ee
     Now \eqref{s21} follows from \eqref{s21a} -\eqref{s22}. Therefore,
     \bel{s24a}
     \sigma_{\rm ess}(K+V) = \sigma_{\rm ess}(K+PVP).
     \ee
     Further, we have $K + PVP = T_V \oplus L$ in  $L^2(\Omega) = \harmo \oplus \harmo^\perp$, and, hence,
     \bel{s25}
     \sigma_{\rm ess}(K+PVP) = \sigma_{\rm ess}(T_V) \cup \sigma_{\rm ess}(L).
     \ee
     By Proposition \ref{p1} (i), we have $\sigma_{\rm ess}(T_V) = V(\partial \Omega)$, and by Proposition \ref{p10},
     $\sigma_{\rm ess}(L) = \emptyset$. Thus, \eqref{s24a} and \eqref{s25} imply \eqref{s23a}.
    \end{proof}
    In the rest of the section we assume that $0 \leq V \in C(\overline{\Omega})$ with $V_{|\partial \Omega} = 0$, and investigate the asymptotic distribution of the discrete spectrum of the operators $K \pm V$, adjoining the origin. For $\lambda > 0$ set
    $$
    {\mathcal N}_-(\lambda) : = {\rm Tr}\,\one_{(-\infty, -\lambda)}(K-V).
    $$
    Set $\lambda_0 : = \inf \sigma(L)$. By Proposition \ref{p10}, we have $\lambda_0 > 0$. For $\lambda \in (0,\lambda_0)$ set
    $$
    {\mathcal N}_+(\lambda) : = {\rm Tr}\,\one_{(\lambda, \lambda_0)}(K+V).
    $$
    Define the compact operator
    $$
    R := {\rm u}-\lim_{\lambda \to 0} Q(K+\lambda)^{-1},
    $$
    where as above $Q = I -P$.
    \begin{theorem} \label{t3}
    Assume that  $0 \leq V \in C(\overline{\Omega})$ and $V_{|\partial \Omega} = 0$.\\
    (i) For any $\eps \in (0,1)$ and $\lambda >0$ we have
    \bel{s26}
    n_+(\lambda; T_V) \leq {\mathcal N}_-(\lambda) \leq n_+((1-\eps)\lambda; T_V) + n_+(\eps; V^{1/2} R V^{1/2}).
    \ee
    (ii) There exist constants $\lambda_1 \in (0,\lambda_0)$ and $C \in [0,\infty)$ such that for any $\eps>0$ and $\lambda \in (0,\lambda_1)$ we have
     \bel{s27}
    n_+((1+\eps)\lambda; T_V) - n_+(\eps; V^{1/2} Q(K-\lambda_1)^{-1} V^{1/2})  \leq {\mathcal N}_+(\lambda) + C \leq n_+(\lambda; T_V).
    \ee
    \end{theorem}
    \begin{proof}
    (i) By the Birman-Schwinger principle \cite[Lemma 1.1]{bir}, we have
    \bel{s28}
    {\mathcal N}_-(\lambda) = n_+(1; (K+\lambda)^{-1/2} V (K+\lambda)^{-1/2}) = n_+(1; V^{1/2} (K + \lambda)^{-1} V^{1/2}), \; \lambda > 0.
    \ee
    It follows from the mini-max principle that
    $$
    n_+(1; (K+\lambda)^{-1/2} V (K+\lambda)^{-1/2}) \geq
    $$
    $$
    n_+(1; P(K+\lambda)^{-1/2} V (K+\lambda)^{-1/2} P) = n_+(\lambda; PVP) = n_+(\lambda; T_V),
    $$
    which, combined with the first equality in \eqref{s28}, implies the lower bound in \eqref{s26}. Further, by the Weyl inequalities \eqref{23} and the elementary identity
    \bel{s50}
    n_+(s;V^{1/2} P (K+\lambda)^{-1} V^{1/2}) = n_+(s\lambda; V^{1/2} P V^{1/2}), \quad s>0, \quad \lambda > 0,
    \ee
    we have
    \bel{s29}
    n_+(1; V^{1/2} (K + \lambda)^{-1} V^{1/2}) \leq n_+((1-\eps)\lambda;  V^{1/2} P V^{1/2}) + n_+(\eps; V^{1/2} Q(K + \lambda)^{-1} V^{1/2}), \; \lambda > 0.
    \ee
    Evidently,
    \bel{s30}
    n_+(s;  V^{1/2} P V^{1/2}) = n_+(s; PVP) = n_+(s; T_V), \quad s>0,
    \ee
    while the mini-max principle easily implies that for any $\eps > 0$ we have
    \bel{s31}
    n_+(\eps; V^{1/2} Q(K + \lambda)^{-1} V^{1/2}) \leq  n_+(\eps; V^{1/2} R V^{1/2}), \quad \lambda > 0.
    \ee
    Putting together \eqref{s28} and \eqref{s29} -- \eqref{s31}, we obtain the upper bound in \eqref{s26}.\\
    (ii) Since $\lambda_0 \not \in \sigma_{\rm ess}(K + V)$ the spectrum of $K+V$ cannot accumulate at $\lambda_0$. Hence there exists $\lambda_1 \in (0,\lambda_0)$ such that $(\lambda_1,\lambda_0) \cap \sigma(K+V) = \emptyset$, and therefore
    $$
     {\mathcal N}_+(\lambda)  = {\rm Tr}\,\one_{(\lambda, \lambda_1]}(K+V), \quad \lambda \in (0,\lambda_1).
     $$
    By the generalized Birman-Schwinger principle (see e.g. \cite[Theorem 1.3]{adh}),
    $$
    {\rm Tr}\,\one_{(\lambda, \lambda_1]}(K+V) = n_-(1; V^{1/2} (K - \lambda)^{-1} V^{1/2}) - n_-(1; V^{1/2} (K - \lambda_1)^{-1} V^{1/2})
    $$
    which implies
    \bel{s32}
    {\mathcal N}_+(\lambda) + C = n_-(1; V^{1/2} (K - \lambda)^{-1} V^{1/2}), \quad  \lambda \in (0,\lambda_1),
    \ee
    with $C: = n_-(1; V^{1/2} (K - \lambda_1)^{-1} V^{1/2})$.
    By the Weyl inequalities and the identity
    \bel{s51}
    n_-(s;V^{1/2} P (K-\lambda)^{-1} V^{1/2}) = n_+(s\lambda; V^{1/2} P V^{1/2}), \quad s>0, \quad \lambda \in (0,\lambda_0),
    \ee
    which is analogous to \eqref{s50}, we have
     \bel{s33a}
    n_-(1; V^{1/2} (K - \lambda)^{-1} V^{1/2}) \geq n_+((1+\eps)\lambda;  V^{1/2} P V^{1/2}) - n_+(\eps; V^{1/2} Q(K - \lambda)^{-1} V^{1/2}).
    \ee
    Since the mini-max principle easily implies
    $$
    n_+(\eps; V^{1/2} Q(K - \lambda)^{-1} V^{1/2}) \leq n_+(\eps; V^{1/2} Q(K-\lambda_1)^{-1} V^{1/2}), \quad \lambda \in (0,\lambda_1),
    $$
    we find that \eqref{s32}, \eqref{s33a}, and \eqref{s30}, yield the lower bound in \eqref{s27}. Finally, by the mini-max principle, \eqref{s51}, and \eqref{s30}, we have
    $$
    n_-(1; V^{1/2} (K - \lambda)^{-1} V^{1/2}) \leq n_-(1; V^{1/2} P(K - \lambda)^{-1} V^{1/2})  = n_+(\lambda; T_V),
    $$
    which together with \eqref{s32}, implies the upper bound in \eqref{s27}.
    \end{proof}

    Combining Theorem \ref{t3} and the results of Section \ref{ss21}, \ref{symmetry}, and \ref{ss23}, we could obtain rich information concerning the spectrum of the operator $K \pm V$, adjoining the origin.
   For example, estimates \eqref{s26} -- \eqref{s27} and Theorem \ref{t1} yield  the following result:

\begin{follow} \label{f3}
    Assume that $V$ satisfies \eqref{rez1} with $\gamma > 0$. Then we have
   \bel{s40}
     {\mathcal C} \lambda^{-\frac{d-1}{\gamma}} + O(\lambda^{-\frac{d-2}{\gamma}}) \leq {\mathcal N}_-(\lambda) \leq {\mathcal C} \lambda^{-\frac{d-1}{\gamma}} + O(\lambda^{-\frac{d-1}{\gamma}\kappa}),
     \ee
     \bel{s41}
     {\mathcal C} \lambda^{-\frac{d-1}{\gamma}} + O(\lambda^{-\frac{d-1}{\gamma}\kappa}) \leq {\mathcal N}_+(\lambda) \leq {\mathcal C} \lambda^{-\frac{d-1}{\gamma}} + O(\lambda^{-\frac{d-2}{\gamma}}),
     \ee
     where ${\mathcal C}$ is the constant defined in \eqref{a3a}, while $\kappa = \frac{d}{d+2}$ if $2 \leq d \leq 4$, and $\kappa =  \frac{d-2}{d-1}$ if  $d \geq 4$.
     \end{follow}
     \begin{proof}
     First,  the mini-max principle and the Birman-Schwinger principle entail
     \bel{s42}
     n_+(\eps; V^{1/2} R V^{1/2}) \leq n_+(\eps; V^{1/2} Q(K-\lambda_1)^{-1} V^{1/2}) \leq {\rm Tr}\,\one_{(-\infty, E)}(L),
     \ee
     with $E = \lambda_1 + \eps^{-1} \max_{x \in \overline{\Omega}} V(x)$, $\lambda_1$ being introduced in the statement of Theorem \ref{t3} (ii). It follows from the results of \cite{g2}
     that
     \bel{s43}
     {\rm Tr}\,\one_{(-\infty, E)}(L) = O(E^{d/2}), \quad E \to \infty.
     \ee
     Now pick $\eps = \lambda^{\theta}$ with appropriate $\theta>0$ to be fixed later. Then  \eqref{s26} -- \eqref{s27}, \eqref{a2}, and \eqref{s42} - \eqref{s43} yield
     $$
     {\mathcal C} \lambda^{-\frac{d-1}{\gamma}} + O(\lambda^{-\frac{d-2}{\gamma}}) \leq {\mathcal N}_-(\lambda) \leq {\mathcal C} \lambda^{-\frac{d-1}{\gamma}}  + O(\lambda^{-\frac{d-2}{\gamma}}) + O(\lambda^{-\frac{d-1}{\gamma} + \theta}) + O(\lambda^{-\frac{d}{2}\theta}),
     $$
     $$
     {\mathcal C} \lambda^{-\frac{d-1}{\gamma}}  + O(\lambda^{-\frac{d-2}{\gamma}}) + O(\lambda^{-\frac{d-1}{\gamma} + \theta}) + O(\lambda^{-\frac{d}{2}\theta}) \leq {\mathcal N}_+(\lambda)   \leq {\mathcal C} \lambda^{-\frac{d-1}{\gamma}} + O(\lambda^{-\frac{d-2}{\gamma}}).
     $$
     Picking $\theta = \frac{2(d-1)}{\gamma(d+2)}$, we arrive at \eqref{s40}--\eqref{s41}.
     \end{proof}
     Similarly, estimates \eqref{s26} -- \eqref{s27} with $ \eps \in (0,1)$ fixed, and Proposition \ref{p6} entail

\begin{follow} \label{f2} Let $\Omega = B_1 \subset \re^d$, $d \geq 2$, $0 \leq V \in C(\overline{B}_1)$. Assume that   ${\rm supp}\,V = \overline{B_c}$ for some $c \in (0,1)$, and that for any $\delta \in (0,c)$ we have ${\rm inf}_{x \in B_\delta} V(x) > 0$. Then
  $$
   \lim_{{\lambda} \downarrow 0} |\ln{\lambda}|^{-d+1} \, {\mathcal N}_\pm(\lambda) = \frac{2^{-d+2}}{(d-1)!|\ln{c}|^{d-1}}.
   $$
   \end{follow}

\vspace{1cm}

{\bf Acknowledgements.} Considerable parts of this work have been done during the second author's visit to the University of Bordeaux, France, in 2014 and 2017, and to the Institute of Mathematics, Bulgarian Academy of Sciences in 2015 and 2016.
He thanks these institutions for hospitality and financial support. \\
Both authors gratefully acknowledge the partial
support of the French Research Project ANR-2011-BS01019-01 and of the Chilean Scientific Foundation {\em Fondecyt}
under Grant 1170816.\\

\end{document}